\documentclass[draft,reqno]{amsproc}
\usepackage[utf8]{inputenc}
\usepackage{amssymb}
\usepackage{euscript}
\usepackage{extarrows}
\usepackage{mathrsfs}
\usepackage{a4wide}
\usepackage{setspace}

\makeatletter
\@namedef{subjclassname@2010}{%
\textup{2010} Mathematics Subject Classification}
\makeatother

\newtheorem{thm}{Theorem}
\newtheorem{cor}[thm]{Corollary}
\newtheorem{lem}[thm]{Lemma}
\newtheorem{pro}[thm]{Proposition}

\theoremstyle{remark}
\newtheorem{rem}[thm]{Remark}

\theoremstyle{definition}
\newtheorem{exa}[thm]{Example}

\DeclareMathOperator{\D}{d\hspace{-0.25ex}}

\newcommand*{\ascr}{\mathscr A}
\newcommand*{\at}[1]{{\mathsf{At}} (#1)}

\newcommand*{\cbb}{\mathbb C}
\newcommand*{\card}[1]{{\rm card} (#1)}

\newcommand*{\cfw}{C_{\phi,w}}
\newcommand*{\esf}{\mathsf{E}}
\newcommand*{\efw}{\mathsf{E}_{\phi,w}}

\newcommand*{\dz}[1]{{\EuScript D}(#1)}

\newcommand*{\escr}{\mathscr E}

\newcommand*{\hh}{\mathcal H}

\newcommand*{\hfw}{{\mathsf h}_{\phi,w}}
\newcommand*{\hsf}{{\mathsf h}}

\newcommand*{\is}[2]{\langle#1,#2\rangle}
\newcommand*{\jd}[1]{\EuScript N(#1)}
\newcommand*{\ns}[1]{\EuScript C(#1)}

\newcommand*{\nbb}{\mathbb N}

\newcommand*{\ob}[1]{{\EuScript R}(#1)}

\newcommand*{\rbb}{\mathbb R}
\newcommand*{\rbop}{{\overline{\rbb}_+}}
\newcommand*{\smalloplus}{\raise0pt\hbox{$\scriptscriptstyle \oplus$}}

\newcommand*{\zbb}{\mathbb Z}

\newcommand*{\mpi}[1]{{#1}^\dag}
\newcommand*{\mpcfw}{C_{\phi,w}^\dag}
\begin{document}
\setstretch{1.2}
\title[Reciprocal of weighted composition operators]{Reciprocals of weighted composition operators in $L^2$-spaces}

\author[P.\ Budzy\'{n}ski]{Piotr Budzy\'{n}ski}
\address{Katedra Zastosowa\'{n} Matematyki, Uniwersytet Rolniczy w Krakowie, ul.\ Balicka 253c, 30-198 Kra\-k\'ow, Poland}
\email{piotr.budzynski@urk.edu.pl}

\maketitle
\begin{abstract}
Various formulas for reciprocals of densely defined weighted composition operators in $L^2$-spaces as well as for their adjoints are provided. The relation between the reciprocal of a weighted composition operator and the product of the reciprocals of the associated multiplication operator and composition operator are studied.
\end{abstract}

\section{Introduction}
In \cite{hes-pjm-1961}, Hestenes developed a spectral theory for a closed operators acting between Hilbert spaces, an analogue of the one known for selfadjoint operators. It is based on a generalization of the notion of selfadjointness (i.e., ralative selafadjointness). This in turn relied on the concepts of a reciprocal and $*$-reciprocal of a closed operators.

The reciprocal is one of the many generalized inverses that can be defined for a linear operator (see \cite{ben-gre, wan-wei-qia}). Historically, the first one (the pseudoinverse) was introduced by Fredholm (see \cite{fre}) in the context of integral operators. Another class of operators for which generalized inverses were defined and studied is the one composed of differential operators (see \cite{rei}). For finite matrices the best known general inverse is the Moore-Penrose one (see \cite{moo, pen}). Their importance stems from the various applications they have, for example in the solvability of linear systems.

For weighted composition operators in $L^2$-spaces the formula for the Moore-Penrose inverse was given by Jabbarzadeh and Jafari Bakhshkandi (see \cite{jab-bak-bbmsss}), however it was done in an limited context - the operators considered in the paper were bounded and all could be represented as products of  multiplication operators and composition operators (see \cite[Chapters 6 and 7]{b-j-j-sW} for the discussion of the limitation of this approach). The authors also calculated the stability constant (studied in \cite{jab-bims} as well).

Motivated by the above mentioned research we study reciprocals for closed weighted composition operators (without any other constraints) and their adjoints in this paper. We give formulas for them, general ones (see Theorem 1 and Proposition 3) and the ones that are valid in more specified situations (see Corrolary 6 and 7, and Proposition 8). We also study an interesting problem of comparing the reciprocals of a weighted composition operator and of the associated multiplication and composition operators (see Section 4). The methods and results are different than in \cite{jab-bak-bbmsss}. Further studies (in particular, on spectral properties and generalized Cauchy dual) is to be carried.
\section{Preliminaries}
We write $\zbb$, $\rbb$, and $\cbb$ for the sets of integers, real numbers, and complex numbers, respectively. We denote by $\nbb$, $\zbb_+$, and $\rbb_+$ the sets of positive integers, nonnegative integers, and nonegative real numbers, respectively. Set $\rbop = \rbb_+ \cup \{\infty\}$.

Let $\hh$ be a (complex) Hilbert space and $A$ be an (linear) operator in $\hh$. We denote by $\dz{A}$, $\jd{A}$, $\ob{A}$ the domain, the kernel, the range of $A$, respectively. $A^*$ stands for the adjoint of $A$ whenever it exists.

Let $(X,\ascr, \mu)$ be a $\sigma$-finite measure space and $\phi$ be an $\ascr$-measurable {\em transformation} of $X$ (i.e., an $\ascr$-measurable mapping $\phi\colon X\to X$). Let $w\colon X\to\cbb$ be $\ascr$-measurable. Define measures $\mu_w$ and $\mu_w\circ\phi^{-1}$ on $\ascr$ by $\mu_w(\sigma)=\int_\sigma|w|^2\D\mu$ and $\mu_w\circ \phi^{-1}(\sigma)=\mu_w\big(\phi^{-1}(\sigma)\big)$ for $\sigma\in\ascr$. Assume that $\mu_w\circ\phi^{-1}$ is absolutely continuous with respect to  $\mu$. Then the operator 
\begin{align*}
\cfw \colon L^2(\mu) \supseteq \dz{\cfw} \to L^2(\mu)    
\end{align*}
given by
\begin{align*}
\dz{\cfw} = \{f \in L^2(\mu) \colon w \cdot (f\circ \phi) \in L^2(\mu)\},\quad
\cfw f  = w \cdot (f\circ \phi), \quad f \in \dz{\cfw},
\end{align*}
is well-defined (see \cite[Proposition 7]{b-j-j-sW}); throughout the paper, as usual, $L^2(\mu)$ stands for the complex Hilbert space of all square $\mu$-integrable $\ascr$-measurable complex functions on $X$ (with the standard inner product). We call $\cfw$ the {\em weighted composition operator} induced by the transformation $\phi$ and the {\em weight function} $w$. By the Radon-Nikodym theorem (cf.\ \cite[Theorem 2.2.1]{ash}), there exists a unique (up to a set of $\mu$-measure zero) $\ascr$-measurable function $\hfw\colon X \to \rbop$ such that
 \begin{align} \label{l1-new}
\mu_w \circ \phi^{-1}(\varDelta) = \int_{\varDelta} \hfw \D \mu, \quad
\varDelta \in \ascr.
\end{align}
It follows from \cite[Theorem 1.6.12]{ash} and \cite[Theorem 1.29]{rud} that for every $\ascr$-measurable function $f \colon X \to \rbop$ (or for every $\ascr$-measurable function $f\colon X \to \cbb$ such that $f\circ \phi \in L^1(\mu_w)$),
\begin{align} \label{l2-new}
\int_X f \circ \phi \D\mu_w = \int_X f \, \hfw \D \mu.
\end{align}
Assume that $\cfw$ is densely defined. Then for a given $\ascr$-measurable function $f\colon X\to \rbop$ (or $f\colon X\to \cbb$ such that $f\in L^p(\mu_w)$ with some $1\leqslant p<\infty$) one may consider $\efw(f)$, the conditional expectation of $f$ with respect to the $\sigma$-algebra $\phi^{-1}(\ascr)$ and the measure $\mu_w$ (see \cite[Section 2.4]{b-j-j-sW}); recall that the formal expression $\efw(f)\circ\phi^{-1}$ stands for a unique (up to sets of $\mu$-measure zero) $\ascr$-measurable function on $X$ such that $\efw(f)\circ\phi^{-1}=0$ a.e. $[\mu]$ on $\{\hfw=0\}$ and $\efw(f)=\big(\efw(f)\circ\phi^{-1}\big)\circ\phi$ a.e. $[\mu_w]$.

An important class of weighted composition operators is that formed by composition operators: weighted composition operators with the weight function equal to $1$ a.e. $[\mu]$, denoted by ${\mathbf 1}$. When talking of them we will use simplified notation: $C_\phi:=C_{\phi, {\mathbf 1}}$, $\hsf_\phi:=\hsf_{\phi, {\mathbf 1}}$, and $\esf_\phi:=\esf_{\phi, {\mathbf 1}}$.

The reader is referred to \cite{b-j-j-sW} for the most recent and comprehensive account on (bounded and unbounded) weighted composition operators in $L^2$-spaces.


\section{The reciprocal of $\cfw$}
Let $\hh$ be a Hilbert space and $T$ be a linear operator in $\hh$ such that $\dz{T}=\jd{T}\oplus \ns{T}$, where $\ns{T}=\dz{T}\cap \jd{T}^\perp$. Then there exists a (unique) densely defined operator $\mpi{T}$ defined on $\dz{\mpi{T}}=\ob{T}\oplus {\ob{T}}^\perp$ that satisfies
\begin{align*}
    T\mpi{T}f&=P_{\overline{\ob{T}}}f, \quad f\in\dz{\mpi{T}},\\
    \mpi{T}Tg&=P_{{\jd{T}^\perp}}g, \quad g\in\dz{{T}},
\end{align*}
and
\begin{align*}
    \jd{\mpi{T}}={\ob{T}}^\perp.
\end{align*}
The operator $T^\dag$ is called a {\em reciprocal} of $T$ (see \cite{hes-pjm-1961}) or a {\em maximal Tseng inverse} of $T$ (see \cite[Chapter 9]{ben-gre}). We shall use the former name. In case $T$ is bounded and has a closed range, $T^\dag$ coincides with the Moore-Penrose inverse of $T$. 

The condition $\dz{T}=\jd{T}\oplus \ns{T}$ is satisfied for $T$ that have a closed kernel, in particular for a closed $T$. If $T$ is closed, then $T^\dag$ is closed too. Clearly, since each weighted composition operator is closed, the reciprocal $\mpcfw$ of $\cfw$ exists whenever $\cfw$ is well-defined. 

The following theorem gives a description of the reciprocal for a densely defined weighted composition operator. Here and later on, we use the notation\footnote{We follow standard measure-theoretic conventions: $\frac{1}{0}=\infty$ and $0\cdot \infty =0$.}: for $f\in L^2(\mu)$, $f_w:=\chi_{\{w\neq 0\}}\frac{f}{w}$.
\begin{thm}\label{bazant}
Assume $\cfw$ is densely defined. Then the following conditions hold:
\begin{enumerate}
\item[(i)] $\dz{\mpcfw}=\{f\in L^2(\mu) \colon \efw(f_w)\circ\phi^{-1}\in L^2(\mu)\}$,
\item[(ii)] $\mpcfw f=\efw(f_w)\circ\phi^{-1}$ for $f\in\dz{\mpcfw}$.
\end{enumerate}
\end{thm}
\begin{proof}
Let $\tilde C$ be an operator in $L^2(\mu)$ defined by
\begin{align*}
\dz{\tilde C}&=\{f\in L^2(\mu) \colon \efw(f_w)\circ\phi^{-1}\in L^2(\mu)\},\\
\tilde C f&=\efw(f_w)\circ\phi^{-1},\quad f\in\dz{\mpcfw}.
\end{align*}
Clearly, $\tilde C$ is a well-defined in operator in $L^2(\mu)$. 

Take $f\in \dz{\cfw}$. Then we have
\begin{align*}
(\cfw f)_w=\chi_{\{w\neq0\}}\frac{w \cdot (f\circ\phi)}{w}=\chi_{\{w\neq0\}}\cdot(f\circ \phi)\quad \text{a.e. }[\mu],
\end{align*}
and thus (see \eqref{valencia} below; cf \cite[Section 2.4]{b-j-j-sW})
\begin{align*}
\efw\big( (\cfw f)_w \big)\circ \phi^{-1}&=\efw\big( \chi_{\{w\neq0\}}\cdot(f\circ \phi) \big)\circ \phi^{-1}\\
&=f \cdot \efw\big(\chi_{\{w\neq0\}}\big)\circ\phi^{-1}\quad \text{a.e. }[\mu].
\end{align*}
In view of \cite[Proposition 15]{b-j-j-sW} and $\chi_{\{w\neq 0\}}\leq 1$ a.e. $[\mu_w]$, we have $\efw\big(\chi_{\{w\neq0\}}\big)\circ\phi^{-1}\leq 1$ a.e. $[\mu]$. Hence $\efw\big( (\cfw f)_w \big)\circ \phi^{-1}\in L^2(\mu)$. Since $f$ was arbitrary, we get $\dz{\cfw}\subseteq \dz{\tilde C \cfw}$. Moreover, for any $f\in\dz{\cfw}$ we have
\begin{align*}
    \tilde C\cfw f=f \cdot \efw\big(\chi_{\{w\neq0\}}\big)\circ\phi^{-1}.
\end{align*}
We note that 
\begin{align}\label{valencia}
\text{$\chi_{\{\hfw>0\}}=\efw\big(\chi_{\{w\neq0\}}\big)\circ\phi^{-1}\quad$ a.e. $[\mu]$.}    
\end{align}
This can be deduced in the following way. First, by definition, we observe that \eqref{valencia} is equivalent to
\begin{align}\label{valencia+}
\chi_{\{\hfw>0\}}\circ\phi=\efw\big(\chi_{\{w\neq0\}}\big)\quad \text{a.e. $[\mu]$}.    
\end{align}
Clearly, the left hand side in \eqref{valencia+} is $\phi^{-1}(\ascr)$-measurable, so proving \eqref{valencia+} amounts to showing that
\begin{align}\label{valencia++}
\int_{\phi^{-1}(\varDelta)}\chi_{\{w\neq 0\}}\D\mu_w=
\int_{\phi^{-1}(\varDelta)}\chi_{\{\hfw> 0\}}\circ\phi\D\mu_w,\quad \varDelta\in\ascr.
\end{align}
For any $\varDelta\in\ascr$, we have
\begin{align*}
\int_{\phi^{-1}(\varDelta)}\chi_{\{w\neq 0\}}\D\mu_w=
\int_{\phi^{-1}(\varDelta)}\D\mu_w=\mu_w\big(\phi^{-1}(\varDelta)\big)
\end{align*}
and
\begin{align*}
\int_{\phi^{-1}(\varDelta)}\chi_{\{\hfw> 0\}}\circ\phi\D\mu_w&=
\int_{X}\chi_{\varDelta\cap\{\hfw> 0\}}\circ\phi\D\mu_w
=\int_X \chi_{\varDelta\cap\{\hfw> 0\}} \hfw\D\mu\\
&=\int_X \chi_{\varDelta} \hfw\D\mu=\mu_w\big(\phi^{-1}(\varDelta)\big).
\end{align*}
Thus \eqref{valencia++} holds, and so \eqref{valencia+} holds too. Applying \eqref{valencia} and \cite[Lemma 11]{b-j-j-sW}, we get 
\begin{align}\label{dentysta}
    \tilde C\cfw f=f \cdot \chi_{\{\hfw>0\}}=P_{\jd{\cfw}^\perp}f.
\end{align}

Now, take $f\in \dz{\tilde C}$. Then $f$ and $\efw(f_w)\circ\phi^{-1}$ belong to $L^2(\mu)$. Moreover, we have
\begin{align*}
\int_X\big|\efw(f_w)\circ\phi^{-1}\big|^2\hfw\D\mu&=
\int_X\big|\efw(f_w)\circ\phi^{-1}\big|^2\circ\phi\D\mu_w\\
&=\int_X\big|\efw(f_w)\big|^2\D\mu_w<+\infty,
\end{align*}
where the last inequality holds due to the boundedness of the operator $f\mapsto \efw(f_w)$ (see \cite[Section 2.5]{b-j-j-sW}). By \cite[Proposition 8\,(i)]{b-j-j-sW}, the above (with $f$ being arbitrary) implies that $\dz{\tilde C}\subseteq \dz{\cfw \tilde C}$. We also have
\begin{align*}
\cfw\tilde C f= w\efw(f_w).
\end{align*}
Now we write $f=f^++f^-$ with $f^+:=f-w\efw(f_w)$ and $f^-:=w\efw(f_w)$. Then
\begin{align*}
f^+_w=\chi_{\{w\neq0\}}\frac{f-w\efw(f_w)}{w}=f_w-\chi_{\{w\neq0\}}\efw(f_w)=f_w-\efw(f_w)\quad \text{a.e. $[\mu_w]$}.
\end{align*}
Since $\efw$ is an orthogonal projection (on a closed subspace of $L^2(\mu_w)$) we see that
\begin{align*}
\efw\big(f^+_w\big)=\efw\big(f_w-\efw(f_w)\big)=0.
\end{align*}
Applying \cite[Proposition 17\,(iii)]{b-j-j-sW} we get $f^+\in \jd{\cfw^*}$. On the other hand, we see that\allowdisplaybreaks
\begin{align*}
\is{f^-}{g}&=\int_Xw\efw(f_w)\bar{g}\D\mu
=\int_Xw\efw(f_w)\chi_{\{w\neq0\}}\bar{w}\frac{\bar{g}}{\bar w}\D\mu\\
&=\int_X\efw(f_w)\overline{g_w}\D\mu_w=\int_X f_w\overline{\efw(g_w)}\D\mu_w=0,\quad g\in\jd{\cfw^*},
\end{align*}
which implies that $f^-\perp\jd{\cfw^*}$. Employing $L^2(\mu)=\jd{\cfw^*}\oplus\overline{\ob{\cfw}}$, we see that
\begin{align}\label{dentysta+}
\cfw\tilde C f= w\efw(f_w)=P_{\overline{\ob{\cfw}}}f.
\end{align}

Concluding the proof, we note that given $f\in L^2(\mu)$ we have 
\begin{align*}
f\in\jd{\tilde C} &\iff \efw(f_w)\circ\phi^{-1}=0\quad \text{a.e. $[\mu]$}\\
&\iff \efw(f_w)=0\quad \text{a.e. $[\mu_w]$}\\
&\iff f\perp\ob{\cfw}.
\end{align*}
This together with \eqref{dentysta} and \eqref{dentysta+} yield $\tilde C=\mpcfw$.
\end{proof}
\begin{rem}\label{tolek}
Let $\varDelta\in\ascr$ be a set of finite $\mu$ measure such that $\frac{1}{\hfw\circ\phi}<K$ a.e. $[\mu_w]$ on $\varDelta$ for some $K>0$. Set $f:=\chi_\varDelta$. Then
\begin{align*}
\int_X \big|\efw(f_w)\circ\phi^{-1}\big|^2\D\mu&=
\int_{\{\hfw\neq0\}} \frac{1}{\hfw}\big|\efw(f_w)\circ\phi^{-1}\big|^2\hfw\D\mu\\
&=\int_X\frac{1}{\hfw\circ\phi} \big|\efw(f_w)\big|^2\D\mu_w=\int_X \big|\efw\big(\frac{f_w}{\sqrt{\hfw\circ\phi}}\big)\big|^2\D\mu_w\\
&\leq \int_X \big|\frac{f_w}{\sqrt{\hfw\circ\phi}}\big|^2\D\mu_w
=\int_{\{w\neq0\}\cap\varDelta} \frac{1}{\hfw\circ\phi}\D\mu\\
&\leq K\mu(\varDelta).
\end{align*}
This implies that $f\in\dz{\mpcfw}$. Moreover, since $\hfw\circ\phi>0$ a.e. $[\mu_w]$ and $\mu$ is $\sigma$-finite, the collection of such $f$'s (with $K$ depending on $f$) is linearly dense in $L^2(\mu)$.
\end{rem}

It is well known that the adjoint of the reciprocal of an operator $T$ is equal to the reciprocal of $T^*$. Hence we get the following.
\begin{pro}\label{ruczaj}
Assume $\cfw$ is densely defined. Then the following conditions hold:
\begin{enumerate}
\item[(i)] $\dz{(\cfw^*)^\dag}=\{f\in L^2(\mu) \colon \frac{w}{\hfw\circ\phi} (f\circ \phi)\in L^2(\mu)\}$,
\item[(ii)] $(\cfw^*)^\dag f=\frac{w}{\hfw\circ\phi} (f\circ \phi)$ for $f\in\dz{(\mpcfw)^*}$,
\item[(iii)] $(\cfw^*)^\dag=M_{\hfw^{-1/2}\circ\phi}U$, where $U$ is the unitary part of the polar decomposition $\cfw=U|\cfw|$.
\end{enumerate}
\end{pro}
\begin{proof}
Since $(\mpcfw)^*=(\cfw^*)^\dag$ it suffices to compute the adjoint of $\mpcfw$. 

Take $f\in L^2(\mu)$. We have
\begin{align}\label{lucek}
\is{f}{\mpcfw g}&=\int_X f\, \overline{\efw(g_w)\circ\phi^{-1}}\D\mu=\int_X \hfw^{-1} f\, \overline{\efw(g_w)\circ\phi^{-1}} \,\hfw\D\mu\notag\\
&=\int_X \frac{f\circ\phi}{\hfw\circ\phi} \overline{\efw(g_w)}\D\mu_w=\int_X \frac{f\circ\phi}{\hfw\circ\phi} \overline{g_w}\D\mu_w\\
&=\int_X w\frac{f\circ\phi}{\hfw\circ\phi}\,\overline{g}\D\mu,\quad g\in\dz{\mpcfw}.\notag
\end{align}
In view of the above, $\escr:=\{f\in L^2(\mu) \colon w \frac{f\circ \phi}{\hfw\circ\phi}\in L^2(\mu)\}\subseteq\dz{(\mpcfw)^*}$ and the equality in (ii) holds for every $f\in\escr$. 

Consider now $f\in \dz{(\mpcfw)^*}$ and set $\eta_f:=w \frac{f\circ \phi}{\hfw\circ\phi}$. Clearly, by \eqref{lucek}, $g\eta_f\in L^1(\mu)$ for every $g\in\dz{\mpcfw}$. Considering $g=\chi_\varDelta$, where $\varDelta\in\ascr$ is as in Remark \ref{tolek}, we see that $\int_X \eta_f\, g\D\mu=\int_X (\mpcfw)^*f\, g\D\mu$. Therefore we can deduce from the moreover part of Remark \ref{tolek} and \cite[Lemma 2]{b-j-j-sW}, $\eta_f=(\mpcfw)^*f$; in particular, $\eta_f\in L^2(\mu)$. Hence, $\dz{(\mpcfw)^*}\subseteq \escr$. Using our previous argument we get that $\dz{(\mpcfw)^*}= \escr$ and the equality in (ii) holds for every $f\in\escr$. Thus (i) and (ii) are satisfied.

It is known (see \cite[Theorem 18]{b-j-j-sW}) that the unitary part $U$ in the polar decomposition of $\cfw$ equals to $C_{\phi,\tilde w}$, where $\tilde w=\frac{w}{\sqrt{\hfw\circ\phi}}$ a.e. $[\mu]$. In view of (i) and (ii), $\dz{M_{\hfw^{-1/2}}U}\subseteq \dz{(\mpcfw)^*}$ and $M_{\hfw^{-1/2}}Uf=(\mpcfw)^*f$ for every $f\in \dz{M_{\hfw^{-1/2}}U}\subseteq \dz{(\mpcfw)^*}$. On the other hand, if $f\in \dz{(\mpcfw)^*}$, then $f\in \dz{M_{\hfw^{-1/2}}U}$ (since $U$ is a bounded operator on $L^2(\mu)$ and $\hfw^{-1/2}Uf\in L^2(\mu)$, the latter meaning that $Uf\in \dz{M_{\hfw^{-1/2}}}$). Hence (iii) holds and the proof is complete.
\end{proof}
\begin{rem}\label{cieplo}
In view of Proposition \ref{ruczaj}, we have $(\cfw^*)^\dag=C_{\phi, \hat w}$, where $\hat w=\frac{w}{\hfw\circ\phi}$. 
\end{rem}
An important subclass of weighted composition operators it that of weighted composition operators over discrete measure spaces (it comes handy in particular when constructing examples). Hence the formula for the reciprocal for these operators is worth knowing.
\begin{exa}\label{hms}
Let $(X,\ascr,\mu)$ be a discrete measure space, i.e.: $\ascr=2^X$, $\card{\at{\mu}}\leqslant \aleph_0$, $\mu\big(X\setminus \mu(\at{\mu})\big)=0$ and $\mu(x)<\infty$ for all $x\in X$, where $\at{\mu}=\{x\in X\colon \mu(x)>0\}$. Then for any $\phi\colon X\to X$ and $w\colon X\to \cbb$, $\cfw$ is well-defined and (see \cite[Proposition 79]{b-j-j-sW})
\begin{align}\label{hpr}
\hfw(x)=\frac{\mu_w(\phi^{-1}(\{x\}))}{\mu(x)},\quad x\in\at{\mu}.
\end{align}
Moreover, for every $f\colon X\to\overline{\rbb}_{+}$
\begin{align}\label{hpr+}
\efw(f)=\frac{\int_{\phi^{-1}(\{x\})}f\D\mu_w}{\mu_w(\phi^{-1}(\{x\}))} \text{ on $\phi^{-1}(\{x\})$, $x\in\Omega:=\at{\mu_w\circ\phi^{-1}}$}.
\end{align}
This implies that
\begin{align*}
\big(\mpcfw f\big)(x)=\frac{1}{\mu_w(\phi^{-1}(\{x\}))} \int_{\phi^{-1}(\{x\})}f\bar w\D\mu,\quad x\in\at{\mu_w\circ\phi^{-1}},
\end{align*}
for all $f\in L^2(\mu)$ such that
\begin{align*}
\sum_{x\in\at{\mu_w\circ\phi^{-1}}} \frac{\mu(\phi^{-1}(\{x\}))}{\big(\mu_w(\phi^{-1}(\{x\}))\big)^2} \bigg(\int_{\phi^{-1}(\{x\})}f\bar w\D\mu\bigg)^2<\infty.
\end{align*}
On the other hand, we have
\begin{align*}
\big((\cfw^*)^\dag f\big)(x)=\frac{w(x)\mu(\phi(x))}{\mu_w(\phi^{-1}(\{\phi(x)\}))} f(\phi(x)),\quad x\in\phi^{-1}(\at{\mu}),
\end{align*}
for all for $f\in L^2(\mu)$ such that
\begin{align*}
    \sum_{x\in\at{\mu}} |f(x)|^2\mu_w(\phi^{-1}(\{x\}))<\infty.
\end{align*}
\end{exa}
Composition operators are another important examples of weighted composition operators. For them, and their adjoints, the reciprocals take simplified forms.
\begin{cor}\label{wieden}
Let $C_\phi$ be densely defined in $L^2(\mu)$. Then the following conditions are satisfied:
\begin{enumerate}
\item[(i)] $\dz{C_\phi}=\{f\in L^2(\mu)\colon \esf_\phi(f)\circ\phi^{-1}\in L^2(\mu)\}$ and $C_\phi^\dag f=\esf_\phi(f)\circ\phi^{-1}\in L^2(\mu)\}$ for $f\in \dz{C_\phi}$,
\item[(ii)] $(C_\phi^*)^\dag =C_{\phi,\hsf_\phi^{-1}\circ\phi}$.
\end{enumerate}
\end{cor}
It is in general not true that well-definiteness of a weighted composition operator $\cfw$ implies well-definiteness of a composition operator $C_\phi$ (see e.g., \cite[Example 102]{b-j-j-sW}). However, if the latter operator is well-defined yet another formula for the reciprocal of the adjoint is available. It follows immediately from Proposition \ref{ruczaj} and \cite[Proposition 116]{b-j-j-sW}.
\begin{cor}
Let $C_\phi$ and $\cfw$ be densely defined in $L^2(\mu)$. Then the following hold:
\begin{align*}
\dz{(\cfw^*)^\dag}&=\{f\in L^2(\mu) \colon  \frac{w\,(f\circ \phi)}{(\hsf_\phi\circ\phi) \esf_\phi(|w|^2)}\in L^2(\mu)\},\\
(C_{\phi,w}^*)^\dag&= \frac{w\,(f\circ \phi)}{(\hsf_\phi\circ\phi) \esf_\phi(|w|^2)},\quad f\in\dz{(\cfw^*)^\dag}.
\end{align*}
\end{cor}
Obtaining analogous formulas for the reciprocal of $\cfw$ is possible under additional assumptions.
\begin{pro}\label{barcelona}
Let $C_\phi$ and $\cfw$ be densely defined in $L^2(\mu)$. Assume that $0<|w|<\beta$ a.e. $[\mu]$ for some $\beta\in(0,+\infty)$. Then the following hold:
\begin{align*}
\dz{\mpcfw}&=\{f\in L^2(\mu)\colon \frac{\esf_\phi(f\bar w)\circ\phi^{-1}}{\esf_\phi(|w|^2)\circ\phi^{-1}}\in L^2(\mu)\},\\
\mpcfw f&=\frac{\esf_\phi(f\bar w)\circ\phi^{-1}}{\esf_\phi(|w|^2)\circ\phi^{-1}},\quad f\in \dz{\mpcfw}.
\end{align*}
\end{pro}
\begin{proof}
Clearly, our assumptions yield $\{f\in L^2(\mu)\}\subseteq\{f\in L^2(\mu_w)\}$. Thus for a fixed $f\in L^2(\mu)$, $\esf_\phi(f\bar w)\circ \phi^{-1}$ makes sense. Moreover, we have
\begin{align}\label{rudy}
\esf_\phi(f \bar w) \circ \phi^{-1} \hsf_\phi=\efw (f_w)\circ\phi^{-1}\hfw\quad \text{a.e. $[\mu]$}.
\end{align}
This follows from
\begin{align*}
\int_\varDelta \efw(f_w)\circ\phi^{-1}\hfw\D\mu&=\int_X\chi_{\varDelta}\circ\phi\, \efw (f_w)\D\mu_w
=\int_X (\chi_\varDelta\circ\phi)\, f_w\D\mu_w\\
&=\int_X (\chi_\varDelta\circ\phi)\, f \bar w\D\mu
=\int_X (\chi_\varDelta\circ\phi)\, \esf_\phi(f \bar w)\D\mu\\
&=\int_\varDelta \esf_\phi(f \bar w)\circ\phi^{-1}\hsf_\phi\D\mu,\quad \varDelta\in\ascr.
\end{align*}
Since $\{\hfw>0\}=\{\hsf_\phi>0\}$ a.e. $[\mu]$ (see \cite[Proposition 121]{b-j-j-sW}), \eqref{rudy} and \cite[Proposition 116]{b-j-j-sW} imply that
\begin{align*}
\efw (f_w)\circ\phi^{-1}=\frac{\esf_\phi(f \bar w) \circ \phi^{-1} \hsf_\phi}{\hfw}=\frac{\esf_\phi(f\bar w)\circ\phi^{-1}}{\esf_\phi(|w|^2)\circ\phi^{-1}}\quad \text{a.e. $[\mu]$.}
\end{align*}
Applying Theorem \ref{bazant} completes the proof.
\end{proof}
\begin{rem}
Concerning Proposition \ref{barcelona}, it should be pointed out that dropping the assumption $\alpha <|w|<\beta$ a.e. $[\mu]$ with $\alpha, \beta\in(0,+\infty)$ may cause that the expression $\esf_\phi(f \bar w)$ does no longer make sense, in which case the description of $\dz{\mpcfw}$ is not valid and the formula for $\mpcfw$ holds on a possibly smaller set (containing $L^2(\mu_w)$) than $\dz{\mpcfw}$.
\end{rem}
\section{Comparing $\mpcfw$, $(M_w C_\phi)^\dag$, and $C_\phi^\dag M_w^\dag$}
It is known that under some additional assumptions (see \cite[Chapter 7]{b-j-j-sW} for details) the equality $\cfw=M_w C_\phi$ holds and thus $\mpcfw=(M_wC_\phi)^\dag$ is satisfied. In view of \cite[Theorem 3.5]{hes-pjm-1961}, for closed densely defined operators $A$ and $B$ in $\hh$ we have $(BA)^\dag=A^\dag B^\dag$ whenever 
\begin{align}\label{bat}
\jd{A^*}=\jd{B}.    
\end{align}
Hence, by \cite[Proposition 17]{b-j-j-sW}, we see that:
\begin{cor}\label{ssrk}
Let $C_\phi$ be densely defined and $\{f\in L^2(\mu)\colon \esf_\phi(f)=0\text{ a.e. }[\mu]\}=\jd{M_w}$. Then
\begin{align}\label{tolling}
(M_wC_\phi)^\dag=C_\phi^\dag M_w^\dag    
\end{align}
In particular, if the weight function $w$ is non-zero a.e $[\mu]$, then $\overline{\ob{C_\phi}}=L^2(\mu)$ is sufficient for the equality \eqref{tolling} to hold.
\end{cor}
The above mentioned sufficient condition is not necessary in general for there are weighted composition operators $\cfw$ that are product $M_wC_\phi$ of an injective $M_w$ and a densely defined $C_\phi$, the equality $\mpcfw=C_\phi^\dag M_w^\dag$ is satisfied but the range of $C_\phi$ is not dense (see Example \ref{hiszpa} below). Considering all the above we see that a problem of finding sufficient and necessary conditions for the equalities 
\begin{align*}
\mpcfw=(M_wC_\phi)^\dag=C_\phi^\dag M_w^\dag    
\end{align*}
to hold arises naturally. We will tackle it in this section.
\begin{exa}\label{hiszpa}
Let $X=\zbb_+\cup \{(-k,1)\colon k\in\nbb\}\cup \{(-k,2)\colon k\in\nbb\}$, $\ascr=2^X$ and $\mu$ be the counting measure on $\ascr$, i.e., $\mu(\{x\})=1$ for every $x\in X$. Let $w\colon X\to\cbb$ satisfy $|w((-1,1))|^2=|w((-1,2))|^2$ and $|w|>\alpha$ a.e. $[\mu]$ for some positive $\alpha\in\rbb$. Finally, let $\phi\colon X\to X$ be given by $\phi(k)=k+1$ for $k\in\zbb_+$, $\phi((-1,1))=\phi((-1,2))=0$, and $\phi(-k,l)=(-k+1,l)$ for $k\in\{2, 3, \ldots\}$ and $l\in\{1,2\}$. In view of \cite[Proposition 111]{b-j-j-sW}, $\cfw=M_w C_\phi$. Consider now any $f\in L^2(\mu)$ such that $f((-1,1))=-f((-1,2))\neq 0$ and $f(x)=0$ for $x\notin\{(-1,1),(-1,2)\}$. Then $\esf_\phi(f)=0$ (see Example \ref{hms}) and thus $f\in \jd{C_\phi^*}$. In particular, the range of $C_\phi$ is not dense in $L^2(\mu)$. To show that $\mpcfw=C_\phi^\dag M_w^\dag$ is satisfied we first observe that $|w|^2=\esf_{\phi}(|w|^2)$ which follows easily from \eqref{hpr+}. Also, by \eqref{hpr} and \cite[Proposition 111]{b-j-j-sW} we have
\begin{align*}
\hsf_{\phi,\hat w}(x)=\sum_{y\in \phi^{-1}(\{x\})}\frac{|w(y)|^2}{(\hfw\circ \phi) (y)}=\sum_{y\in \phi^{-1}(\{x\})}\frac{1}{|(\hsf_\phi\circ \phi) (y)\, w(y)|^2},\quad x\in X,
\end{align*}
and
\begin{align*}
\hsf_{\phi,\hat {\mathbf{1}}}(x)=\sum_{y\in \phi^{-1}(\{x\})}\frac{1}{|(\hsf_\phi\circ \phi) (y)|^2},\quad x\in X
\end{align*}
(see Corollary \ref{konopiste} for the required details). This implies that $\hsf_{\phi,\hat w}\leqslant \frac{1}{\alpha^2} \hsf_{\phi,\hat {\mathbf{1}}}$ and proves that $\mpcfw=C_\phi^\dag M_w^\dag$ (see Corollary \ref{konopiste}).
\end{exa}
The next two examples (Example \ref{hiszpa+} and Example \ref{hiszpa-}) are based on  \cite[Example 143]{b-j-j-sW}.
\begin{exa}\label{hiszpa+}
Let $(X,\ascr)=(\zbb, 2^\zbb)$ and $\mu$ be a a counting measure on $\ascr$. Define $\phi\colon X\to X$ and $w\colon X\to\cbb$ as follows:
\begin{align*}
\phi(n)=\left\{
    \begin{tabular}{cc}
    $n-1$&\text{if $n\leq0$},  \\
    $0$&\text{if $n>1$},
    \end{tabular}
\right.\quad 
w(n)=\left\{
    \begin{tabular}{cc}
    $\frac{1}{n}$&\text{if $n\neq0$},  \\
    $1$&\text{if $n=0$}.
    \end{tabular}
\right.
\end{align*}
Then $\cfw$ is densely defined in $L^2(\mu)$. On the other hand, $C_\phi$ is not densely defined (it is well-defined). Moreover $M_wC_\phi$ is closed and $M_wC_\phi\neq\cfw$. Thus, even though both $C_\phi^\dag$ and $(M_wC_\phi)^\dag$ do exist, Corollary \ref{wieden} is not applicable.
\end{exa}
An example of a pair of operators $M_w$ and $C_\phi$ giving a non-closed $M_wC_\phi$ is given below.
\begin{exa}\label{hiszpa-}
Let $X, \ascr, \phi$ and $w$ be as in Example \ref{hiszpa+}. Let $\mu$ be any measure on $\ascr$ such that $0<\mu(\{n\})\leq1$ for every $n\in\zbb$, $\sum_{k=1}^\infty \mu(\{k\})<\infty$ and $\lim\limits_{n\to-\infty}\frac{\mu(\{k\})}{\mu(\{k+1\})}=0$. Then the operators $\cfw$, $C_\phi$, and $M_wC_\phi$ are densely defined. Moreover, $M_wC_\phi$ is not closed and $M_wC_\phi\neq \cfw$. Clearly, $\mpcfw$ and $C_\phi^\dag$ exist. In view of Proposition \ref{ftch}, $(M_wC_\phi)^\dag$ exist as well (alternatively to using Proposition \ref{ftch}, one can simply observe that $\jd{M_w C_\phi}=\jd{C_\phi}=\overline{\jd{C_\phi}}$, which implies that $\dz{M_wC_\phi}$ can be decomposed as per requirement).
\end{exa}
Sometimes the reciprocals of two operators $A$ and $B$ satisfying $A\subseteq B$ are nicely related.
\begin{lem}\label{brintelix}
Suppose that $A\subseteq B$, the reciprocals $A^\dag$ and $B^\dag$ exist, and the following condition
\begin{align}\label{annaboda}
\dz{A}\cap \jd{A}^\perp \subseteq \jd{B}^\perp
\end{align}
is satisfied. Then $A^\dag=B^\dag|_{\ob{A}}\oplus 0|_{{\ob{A}^\perp}}$.
\end{lem}
\begin{proof}
Due to \eqref{annaboda}, every $g\in\dz{A}$ can be written as $g=g_1+g_2$ with $g_1\in\jd{A}\subseteq \jd{B}$ and $g_2\in \dz{A}\cap \jd{A}^\perp \subseteq \dz{B}\cap \jd{B}^\perp$. In particular, we have 
\begin{align}\label{annaboda1}
P_{\jd{A}^\perp}g=P_{\jd{B}^\perp}g,\quad g\in\dz{A}.    
\end{align}
Define an operator $\hat A$ by the formula 
\begin{align*}
\hat A=B^\dag|_{\ob{A}}\oplus 0|_{{\ob{A}^\perp}}.   
\end{align*}
We show that $\hat A=A^\dag$. Clearly, 
\begin{align}\label{chybie0}
\text{$\dz{\hat A}=\dz{A^\dag}$ and $\jd{\hat A}=\jd{A^\dag}$.}    
\end{align}
For any $f\in \dz{\hat A}$ there are $f_1\in \ob{A}$ and $f_2\in \ob{A}^\perp$ such that $f=f_1+f_2$. Thus, for such $f$'s we have
\begin{align}\label{chybie1}
A\hat A f= A\hat A (f_1+f_2)=A\hat A f_1=B B^\dag f_1=P_{\overline{\ob{B}}}f_1=f_1=P_{\overline{\ob{A}}}f,\quad f\in\dz{\hat A}.
\end{align}
By \eqref{annaboda1} we also have
\begin{align}\label{chybie2}
\hat A A g=B^\dag B g=P_{\jd{B}^\perp}g=P_{\jd{A}^\perp}g,\quad g\in\dz{A}.
\end{align}
Combining \eqref{chybie0}-\eqref{chybie2} we get $\hat A= A^\dag$.
\end{proof}
\begin{rem}
Considering the proof of Lemma \ref{brintelix} we note that \eqref{chybie1} holds without assuming \eqref{annaboda}. Also, we note that if \eqref{annaboda} does not hold, \eqref{annaboda1} does not hold as well. Indeed, let $h\in\dz{A}\cap\jd{A}^\perp$ such that $P_{\jd{B}}h\neq 0$. Put $g_2^A:=h$ and $g_2^B:=h-P_{\jd{B}}h\in\dz{B}\cap\jd{B}^\perp$. Take any nonzero $g_1^2\in \jd{A}$ and set $g_1^B:=g_1^A+h\in \jd{B}$. Then for $g:=g_1^A+g_2^A=g_1^B+g_2^B$ we have $P_{\jd{A}^\perp}g=g_2^A\neq g_2^B=P_{\jd{B}^\perp}$.
\end{rem}
Lemma \ref{brintelix} turns out to be applicable in the context of weighted composition operators. It gives a formula for the reciprocal of $M_wC_\phi$ in terms of the reciprocal of $\cfw$.
\begin{pro}\label{ftch}
Let $C_\phi$ be well-defined. Then $(M_wC_\phi)^\dag=\mpcfw|_{\ob{M_wC_\phi}}\oplus 0|_{\ob{M_wC_\phi}^\perp}$.
\end{pro}
\begin{proof}
First we show that the reciprocal of $M_wC_\phi$ exists. This amounts to showing that
\begin{align}\label{eldwa0}
\dz{M_wC_\phi}=\jd{M_wC_\phi}\oplus \big(\dz{M_wC_\phi}\cap\jd{M_wC_\phi}^\perp\big).
\end{align}
Since $\dz{M_wC_\phi}=L^2((1+\hsf_\phi+\hfw)\D\mu)$ (see \cite[Theorem 112]{b-j-j-sW}) and $\jd{\cfw}=\chi_{\{\hfw=0\}}L^2(\mu)$ (see \cite[Lemma 11]{b-j-j-sW}) we get
\begin{align}\label{eldwa1}
\jd{M_wC_\phi}=\chi_{\{\hfw=0\}}L^2((1+\hsf_\phi+\hfw)\D\mu).    
\end{align}
Hence $\overline{\jd{M_wC_\phi}}=\chi_{\{\hfw=0\}\cap\{\hsf_\phi<\infty\}}L^2(\mu)$ and thus $\jd{M_wC_\phi}^\perp=\chi_{\{\hfw>0\}\cup\{\hsf_\phi=\infty\}}L^2(\mu)$. Consequently, we get
\begin{align}\label{eldwa2}
\dz{M_wC_\phi}\cap\jd{M_wC_\phi}^\perp=\chi_{\{\hfw>0\}}L^2((1+\hsf_\phi+\hfw)\D\mu).
\end{align}
Combining \eqref{eldwa1}, \eqref{eldwa2} and $\dz{M_wC_\phi}=L^2((1+\hsf_\phi+\hfw)\D\mu)$ we get \eqref{eldwa0}.

In view of Lemma \ref{brintelix} it suffices now to check whether 
\begin{align}\label{suchylas}
\dz{M_wC_\phi}\cap\jd{M_wC_\phi}^\perp\subseteq \jd{\cfw}^\perp
\end{align}
holds. This clearly is true due to \eqref{eldwa2} and $\dz{\cfw}\cap\jd{\cfw}^\perp=\chi_{\{\hfw>0\}}L^2((1+\hfw)\D\mu).$
\end{proof}
\begin{rem}
If $M_wC_\phi$ is closed then $\hsf_\phi\leqslant c(1+\hfw)$ a.e. $[\mu]$ on $\{\hsf_\phi<\infty\}$ (see \cite[Theorem 112]{b-j-j-sW}). Thus we get $\dz{M_wC_\phi}=\chi_{\{\hsf_\phi<\infty\}}L^2((1+\hfw)\D\mu)$. Hence
\begin{align*}
\dz{M_wC_\phi}\cap\jd{M_wC_\phi}^\perp=\chi_{\{\hsf_\phi<\infty\}\cap\{\hfw>0\}}L^2((1+\hfw)\D\mu).
\end{align*}
\end{rem}
The following can be easily derived from Proposition \ref{ftch}.
\begin{lem}\label{balicka5}
Let $C_\phi$ be well-defined. Then $(M_w C_\phi)^\dag=\mpcfw$ if and only if $\ob{M_w C_\phi}^\perp = \ob{\cfw}^\perp$ and $\mpcfw f=0$ for every $f\in\ob{M_w C_\phi}^\perp$. 
\end{lem}
In general, we do have only the inclusion $M_wC_\phi\subseteq \cfw$, thus $\ob{\cfw}^\perp\subseteq \ob{M_w C_\phi}^\perp$. Obviously, the equality 
\begin{align}\label{johndog1}
\ob{\cfw}^\perp= \ob{M_w C_\phi}^\perp    
\end{align}
holds if and only if $\overline{\ob{\cfw}}= \overline{\ob{M_w C_\phi}}$.
In case of both $C_\phi$ and $\cfw$ densely defined, this is equivalent to
\begin{align}\label{balicka4}
\jd{C_\phi^*M_w^*}=\jd{\cfw^*}
\end{align}
or, due to \cite[Proposition 17]{b-j-j-sW}, to
\begin{align*}
\{f\in L^2(\mu)\colon \efw(f_w)=0\}=\jd{(M_w C_\phi)^*},
\end{align*}
In the proof of the following theorem we show that \eqref{balicka4} actually holds.
\begin{thm}\label{bt}
Assume that $C_\phi$ and $\cfw$ are densely defined. Then $(M_w C_\phi)^\dag=\mpcfw$.
\end{thm}
\begin{proof}
If $f\in \jd{C_\phi^*M_w^*}$, then either $f=\chi_{\{w=0\}}f$ or $f\in L^2(\mu_w)$ and $\esf_\phi(f\bar w)=0$ a.e. $[\mu]$. This means that
\begin{align*}
\is{f}{wg\circ\phi}=\int_Xf\overline {w g\circ\phi}\D\mu=\int_X\esf_\phi(f\bar{w}) \overline{g\circ\phi}\D\mu=0,\quad g\in \dz{M_wC_\phi}.
\end{align*}
Hence, ${\jd{C_\phi^*M_w^*}}\subseteq\ob{M_wC_\phi}^\perp\cap L^2(\mu_w)=\jd{(M_wC_\phi)^*}\cap L^2(\mu_w)$. On the other hand, if $f\in \ob{M_wC_\phi}^\perp\cap L^2(\mu_w)$, then
\begin{align*}
0=\is{f}{wg\circ\phi}=\int_x \esf_\phi(f\bar w) \overline{g\circ\phi}\D\mu=\int_X \esf_{\phi}(f\bar w)\circ\phi^{-1} \bar g \hsf_{\phi}\D\mu,\quad g\in\dz{M_wC_\phi}
\end{align*}
and this yields $\esf_\phi(f\bar w)\hsf_\phi=0$ a.e. $[\mu]$. Since $\esf_\phi(f\bar w)=\chi_{\{\hsf_\phi<0\}}\esf_\phi(f\bar w)$, we deduce that $\esf(f\bar w)=0$ a.e. $\mu$. Thus $\jd{(M_wC_\phi)^*}\cap L^2(\mu_w)\subseteq {\jd{C_\phi^*M_w^*}}$. Summing up, we proved that 
\begin{align}\label{balicka1}
{\jd{C_\phi^*M_w^*}}=\jd{(M_wC_\phi)^*}\cap L^2(\mu_w).  
\end{align}
We note that \eqref{rudy} is valid for $f\in L^2(\mu)\cap L^2(\mu_w)$ and so $\esf_{\phi}(f_w)=0$ a.e. $[\mu_w]$ implies $\efw(f_w)=0$ a.e. $[\mu_w]$ for such $f$'s. Hence, by \eqref{balicka1}, for any $f\in\jd{(M_wC_\phi)^*}\cap L^2(\mu_w)$, we have $\efw(f_w)=0$ a.e. $[\mu_w]$ and this means that $f\in \jd{\cfw^*}\cap L^2(\mu_w)$. Thus $\jd{(M_wC_\phi)^*}\cap L^2(\mu_w)\subseteq \jd{\cfw^*}\cap L^2(\mu_w)$. Clearly, the opposite inclusion holds and so we have 
\begin{align}\label{balicka2}
\jd{(M_wC_\phi)^*}\cap L^2(\mu_w)= \jd{\cfw^*}\cap L^2(\mu_w).
\end{align}
Combining \eqref{balicka1}, \eqref{balicka2} we see that $\jd{C_\phi^*M_w^*}=\jd{\cfw^*}\cap L^2(\mu_w)\subseteq \jd{\cfw^*}$. This together with $\jd{\cfw^*}\subseteq \jd{(M_wC_\phi)^*}$ gives the equality
\begin{align*}
\jd{C_\phi^*M_w^*}=\jd{\cfw^*}
\end{align*}
Also we see that $\mpcfw{f}=0$ for every $f\in\jd{(M_wC_\phi)^*}$ (by \eqref{rudy} again). Therefore, \eqref{johndog1} holds and $\mpcfw f=0$ for $f\in\ob{M_wC_\phi}^\perp$ is satisfied. Apllying Lemma \ref{balicka5} completes the proof.
\end{proof}
Below we focus on the equality $\mpcfw=C_\phi^\dag M_w^\dag$.
\begin{thm}\label{dt}
Let $\cfw$ and $C_\phi$ be densely defined. Consider the following conditions:
\begin{enumerate}
\item[(a$_1$)] $\mpcfw=C_\phi^\dag M_w^\dag$,
\item[(a$_2$)] $\mpcfw\subseteq C_\phi^\dag M_w^\dag$,
\item[(a$_3$)] $\mpcfw\supseteq C_\phi^\dag M_w^\dag$,
\item[(b$_1$)] for every $f\in \dz{\cfw^\dag}$, $f_w\in L^2(\mu)$ and $\efw(f_w)\circ\phi^{-1}=\esf_\phi(f_w)\circ\phi^{-1}$  a.e. $[\mu]$,
\item[(b$_2$)] for every $f\in L^2(1+|w|^{-2}\D\mu)$, $\efw(f_w)\circ\phi^{-1}=\esf_\phi(f_w)\circ\phi^{-1}$  a.e. $[\mu]$,
\item[(c)]$|w|^2=\esf_\phi(|w|^2)$ a.e. $[\mu]$ on $\{w\neq 0\}$.
\end{enumerate}
Then (a$_1$) $\Leftrightarrow$ (a$_2$) $\Leftrightarrow$ (b$_1$) $\Rightarrow$ (c) $\Rightarrow$ (b$_2$) $\Rightarrow$ (a$_3$). 
\end{thm}
\begin{proof}
That (a$_1$) implies (a$_2$) is trivial. 

Since $M_w^\dag f= f_w$ for $f\in\dz{M_w^\dag}$, that (a$_2$) is equivalent to (b$_1$) follows from Theorem \ref{bazant} and Corollary \ref{wieden}.

Assuming (b$_1$), we get $\efw(f_w)\circ\phi^{-1}\hfw=\esf_\phi(f_w)\circ\phi^{-1}\esf_\phi(|w|^2)\circ\phi^{-1}\hsf_\phi$  a.e. $[\mu]$ for every $f\in \dz{\cfw^\dag}$ by \cite[Proposition 116]{b-j-j-sW}. Then we get
\begin{align*}
\int_X \efw(f_w)\circ\phi^{-1}\hfw\D\mu=\int_X \esf_\phi(f_w)\circ\phi^{-1}\esf_\phi(|w|^2)\circ\phi^{-1}\hsf_\phi \D\mu, \quad f\in \dz{\cfw^\dag},
\end{align*}
and thus\allowdisplaybreaks
\begin{align*}
\int_X f_w |w|^2\D\mu&=
\int_X \efw(f_w)\D\mu_w=\int_X \esf_\phi(f_w)\esf_\phi(|w|^2)\D\mu\\
&=\int_X f_w\esf_\phi(|w|^2)\D\mu,\quad f\in \dz{\cfw^\dag}.
\end{align*}
This implies (c).

Assume (c). We note that for $f\in L^2(1+|w|^{-2}\D\mu)$ both the expressions $\efw(f_w)\circ\phi^{-1}$ and $\esf_\phi(f_w)\circ\phi^{-1}$ make sense. Moreover, we have\allowdisplaybreaks
\begin{align*}
\int_{\varDelta} \efw(f_w)\circ\phi^{-1}\hfw\D\mu&=\int_{\phi^{-1}(\varDelta)} \efw(f_w)\D\mu_w=\int_{\phi^{-1}(\varDelta)} f_w |w|^2\D\mu\\
&=\int_{\phi^{-1}(\varDelta)} f_w\esf_\phi(|w|^2)\D\mu,
=\int_{\phi^{-1}(\varDelta)} \esf_\phi(f_w)\esf_\phi(|w|^2)\D\mu\\
&=\int_{\varDelta} \esf_\phi(f_w)\circ\phi^{-1}\esf_\phi(|w|^2)\circ\phi^{-1}\hsf_\phi\D\mu\\
&=\int_{\varDelta} \esf_\phi(f_w)\circ\phi^{-1}\hfw\D\mu, \quad f\in L^2(1+|w|^{-2}\D\mu),\ \varDelta\in\ascr.
\end{align*}
This implies $\efw(f_w)\circ\phi^{-1}\hfw=\esf_\phi(f_w)\circ\phi^{-1}\hfw$  a.e. $[\mu]$ for $f\in L^2(1+|w|^{-2}\D\mu)$. Since $\{\esf_\phi(\cdot)\circ    \phi^{-1}\neq0\}\cap\{\efw(\cdot)\circ\phi^{-1}\neq 0\}\subseteq \{\hfw \neq 0\}$ (see \cite[Proposition 116]{b-j-j-sW}) we get (b$_2$).

Assume (b$_2$). Suppose $f\in\dz{C_\phi^\dag M_w^\dag}$. Then $f\in\dz{M_w^\dag}$ or, equivalently, $f\in L^2((1+|w|^{-2})\D\mu)$. Thus $\esf_{\phi}(f_w)\circ\phi^{-1}$ makes sense and, since $M_w^\dag f\in\dz{C_\phi^\dag}$, belongs to $L^2(\mu)$. This implies that $\efw(f_w)\circ\phi^{-1}\in L^2(\mu)$, which means $f\in \dz{\mpcfw}$. Moreover, $\mpcfw f= C_\phi^\dag M_w^\dag f$. Hence, (a$_3$) is satisfied.

Since, as we proved above, (a$_2$) yields (a$_3$), we see that (a$_2$) implies (a$_1$). This completes the proof.
\end{proof}
Assuming that $M_w$ is bounded from below enables using the adjoints and showing that the Radon-Nikodym derivatives associated to the adjoints of $C_\phi$ and $\cfw$ need to be related whenever $\mpcfw=(M_wC_\phi)^\dag=C_\phi^\dag M_w^\dag$ hold.
\begin{pro}\label{konopiste}
Let $\cfw$ be densely defined. Assume that $\alpha< |w|$ a.e. $[\mu]$ with some positive $\alpha\in\rbb$. Then $C_\phi$ is densely defined and $\cfw=M_wC_\phi$. Moreover, $(M_wC_\phi)^\dag=C_\phi^\dag M_w^\dag$ if and only if the the following two conditions hold:
\begin{itemize}
    \item[(i)] $\hsf_{\phi,\hat w}\leq \beta\,  \hsf_{\phi, \hat{\mathbf{1}}}$ a.e. $[\mu]$ with a positive $\beta\in\rbb$, where $\hat w=\frac{w}{\hfw\circ \phi}$ and $\hat{\mathbf{1}}=\frac{1}{\hsf_\phi\circ\phi}$,
    \item[(ii)] $|w|^2=\esf_\phi(|w|^2)$ a.e. $[\mu]$.
\end{itemize}
\end{pro}
\begin{proof}
The first part of the claim follows from \cite[Proposition 111]{b-j-j-sW}.

Assume that $(M_wC_\phi)^\dag=C_\phi^\dag M_w^\dag$. By Theorems \ref{dt} and \ref{bt} we get (ii) and $(\mpcfw)^*=(C_\phi^\dag M_w^\dag)^*$. It is known that $B^*A^*\subseteq (AB)^*$ for any operators $A$ and $B$ such that all the expressions make sens, we have $(C_\phi^\dag M_w^\dag)^*\supseteq(M_w^\dag)^*(C_\phi^\dag)^*$. Thus using Proposition \ref{ruczaj} we get
\begin{align}\label{len}
C_{\phi,\hat w}\supseteq  M_{\bar w^{-1}}C_{\phi,\hat{\mathbf{1}}}.
\end{align}
Since $M_{\bar w^{-1}}$ is a bounded operator on $L^2(\mu)$, $\dz{M_{\bar w^{-1}}C_{\phi,\hat{\mathbf{1}}}}=\dz{C_{\phi,\hat{\mathbf{1}}}}$. Thus $L^2((1+\hsf_{\phi,\hat{w}})\D\mu)=\dz{C_{\phi,\hat w}}\supseteq\dz{C_{\phi,\hat{\mathbf{1}}}}=L^2((1+\hsf_{\phi,\hat{\mathbf{1}}})\D\mu)$. Applying \cite[Corollary 12.4]{b-j-j-s-ampa} we deduce (i).

Assume that (i) and (ii) hold. From (i) we deduce $\dz{C_{\phi,\hat w}}\supseteq\dz{C_{\phi,\hat{\mathbf{1}}}}=\dz{M_{\bar w^{-1}}C_{\phi,\hat{\mathbf{1}}}}$. This and (ii)  imply \eqref{len}.  Thus we get $(\mpcfw)^*\supseteq (M_w^\dag)^*(C_\phi^\dag)^*$. Taking the adjoints and using the closedness of the operators in question and the boundedness of $M_w^\dag$, we get 
\begin{align*}
\mpcfw=(\mpcfw)^{**}\subseteq \big((M_w^\dag)^*(C_\phi^\dag)^*\big)^*=(C_\phi^\dag)^{**}(M_w^\dag)^{**}=C_\phi^\dag M_w^\dag   
\end{align*}
Applying Theorem \ref{dt} completes the proof.
\end{proof}
\begin{rem}
We note that, in view of Proposition \ref{konopiste}, the condition (c) of Theorem \ref{dt} alone is not sufficient for the equality $\mpcfw=C_\phi^\dag M_w^\dag$.
\end{rem}
Let us revisit briefly Example \ref{hiszpa}. One can easily calculate that for all $x$'s different from $0$ we have
\begin{align*}
\hsf_{\phi,\hat w}(x)=\frac{1}{|w(\phi^{-1}(x))|^2},\quad \hsf_{\phi,\hat{\mathbf{1}}}(x)=1,
\end{align*}
which suggests that the inequality $\hsf_{\phi,\hat w}\leqslant \beta \hsf_{\phi,\hat{\mathbf{1}}}$ is interconnected with the boundedness from below of $|w|$. 
That this is indeed a case is shown in the following.
\begin{pro}\label{wariat}
Let $(X,\ascr,\mu)$ be a discrete measure space, $\phi\colon X\to X$ and $w\colon X\to \cbb\setminus\{0\}$. Assume that $C_\phi$ and $\cfw$ are densely defined and $\esf_\phi(|w|^2)=|w|^2$ a.e. $[\mu]$. Then the conditions:
\begin{itemize}
\item[(i)] $\alpha<|w|$ a.e. $[\mu]$ with a positive $\alpha\in\rbb$,
\item[(ii)] $\hsf_{\phi,\hat w}\leq \beta\,  \hsf_{\phi, \hat{\mathbf{1}}}$ a.e. $[\mu]$ with a positive $\beta\in\rbb$,
\end{itemize}
are equivalent. Moreover, if any of the two above conditions holds, then $(M_wC_\phi)^\dag=C_\phi^\dag M_w^\dag$.
\end{pro}
\begin{proof}
Using \eqref{hpr} and \cite[Proposition 116]{b-j-j-sW} we can easily show that
\begin{align*}
\hsf_{\phi,\hat w}(x)&=\frac{1}{\mu(x)\hsf_\phi^2(x)}\sum_{y\in\phi^{-1}(\{x\})}\frac{\mu(y)}{|w(y)|^2},\\
\hsf_{\phi, \hat{\mathbf{1}}}(x)&=\frac{1}{\mu(x)\hsf_\phi^2(x)}\sum_{y\in\phi^{-1}(\{x\})}\mu(y),
\end{align*}
for every $x\in\at{\mu}$. Since $|w|^2=\esf_{\phi}(|w|^2)$ is constant on sets on the form $\phi^{-1}(\{x\})$, we see that (ii) holds if and only if $\frac{1}{\esf(|w|^2)\circ\phi^{-1}(x)}<\beta$ for all $x\in\at{\mu}$. In turn, the latter condition holds if and only if $|w|$ is $\mu$-essentially bounded from below, i.e. (i) holds.

The ``moreover'' part of the claim follows from Proposition \ref{konopiste}.
\end{proof}

\bibliographystyle{amsalpha}

\end{document}